\documentclass[11pt]{amsart}
\usepackage{amsmath, amsfonts, amsthm, amssymb, multicol}
\usepackage{graphicx}
\usepackage{float}
\usepackage{verbatim}

\allowdisplaybreaks

\usepackage[square,sort,comma,numbers]{natbib}
\usepackage{fancyhdr}
 
\numberwithin{equation}{section}

\newtheorem{thm}{Theorem}[section]
\newtheorem{lma}[thm]{Lemma}
\newtheorem{cor}[thm]{Corollary}

\renewcommand{\epsilon}{\varepsilon}

\newcommand{\conp}{\textup{con}^+(\mathbb{D}^n)}

\renewcommand{\geq}{\geqslant}
\renewcommand{\leq}{\leqslant}

\newcommand{\ubd}{\overline{\dim}_{\textup{B}}}
\newcommand{\lbd}{\underline{\dim}_{\textup{B}}}
\newcommand{\bd}{\dim_{\textup{B}}}
\newcommand{\hd}{\dim_{\textup{H}}}

\title{Dimensions of Kleinian orbital sets}

\author{Thomas Bartlett \& Jonathan M. Fraser\\ \\
 U\MakeLowercase{niversity of} S\MakeLowercase{t} A\MakeLowercase{ndrews}, S\MakeLowercase{cotland} \\
\MakeLowercase{Emails: tomj.bartlett@gmail.com \& jmf32@st-andrews.ac.uk}}
\thanks{JMF was  financially supported by an \emph{EPSRC Standard Grant} (EP/R015104/1) and a  \emph{Leverhulme Trust Research Project Grant} (RPG-2019-034). The authors thank Liam Stuart for helpful comments. }

\begin{document}


\maketitle
\thispagestyle{empty}

\begin{abstract}
Given a non-empty bounded subset of hyperbolic space and a Kleinian group acting on that space, the \emph{orbital set} is the orbit of the given set under the action of the group. We may view orbital sets as bounded (often fractal) subsets of   Euclidean space.  We prove that  the upper box dimension of an orbital set is given by the maximum of three quantities: the upper box dimension of the given set; the Poincar\'e exponent of the Kleinian group; and the upper box dimension of the limit set of the Kleinian group.  Since we do not make any assumptions about the Kleinian group, none of the terms in the maximum can be removed in general. We show by constructing an explicit example that the (hyperbolic) boundedness assumption on $C$ cannot be removed in general. 
\\ \\ 
\emph{Mathematics Subject Classification} 2010: primary: 37F32, 28A80; secondary:  30F40, 28A78, 11J72.
\\
\emph{Key words and phrases}: orbital set, Kleinian group, Poincar\'e exponent,  upper box dimension, limit set, inhomogeneous attractor.
\end{abstract}

\section{Kleinian orbital sets}

\subsection{Hyperbolic geometry,  Kleinian groups, and orbital sets}

Let  $n \geq 2$ be an integer and consider the Poincar\'e ball 
\[
\mathbb{D}^n = \{ z \in \mathbb{R}^n: |z| <1\}
\]
equipped with the hyperbolic metric $d $ given by
\[
|ds| = \frac{|dz|}{1-|z|^2}.
\]
This provides a model of $n$-dimensional hyperbolic space.  The group of orientation preserving isometries of $(\mathbb{D}^n, d )$ is the group of conformal automorphisms of $\mathbb{D}^n$, which we denote by $\conp$.  A group $\Gamma \leq  \conp$ is called \emph{Kleinian} if it is a discrete   subset of $\conp$.  Kleinian groups generate fractal limit sets living on the boundary $S^{n-1}$ as well as  beautiful tessellations of hyperbolic space.  Both of these objects are defined via orbits.  The \emph{limit set} is defined by
\[
L(\Gamma) = \overline{\Gamma(0)} \setminus \Gamma(0)
\]
where $\Gamma(0) = \{ g(0) : g \in \Gamma\}$ is the orbit of 0 under $\Gamma$ and $\overline{\Gamma(0)}$ is the   Euclidean closure of $\Gamma(0)$.  On the other hand, hyperbolic  tessellations arise by taking the orbit of a fundamental domain for the group action.

The \emph{Poincar\'e exponent} is a coarse measure of  the rate of accumulation to the boundary.  It is defined as the exponent of convergence of the \emph{Poincar\'e series} 
\[
P_\Gamma(s) = \sum_{g \in \Gamma } \exp(-sd (0,g(0))) = \sum_{g \in \Gamma} \left(\frac{1-|g(0)|}{1+|g(0)|} \right)^s
\]
for $s \geq 0$.  Therefore the  \emph{Poincar\'e exponent} may be expressed as 
\[
\delta(\Gamma) = \inf\{ s \geq 0 : P_\Gamma(s) <\infty\}.
\]
A Kleinian group is called \emph{non-elementary} if its limit set contains at least 3 points, in which case it is necessarily an uncountable perfect set.  In the case $n=2$ Kleinian groups are more commonly referred to as Fuchsian groups.  For more background on hyperbolic geometry and Kleinian groups see \cite{beardon, maskit}.

In this paper we introduce and study Kleinian orbital sets.  In some sense these provide a bridge between   limit sets and hyperbolic tessellations.  Fix a non-empty  set $C \subseteq \mathbb{D}^n$ and a Kleinian group $\Gamma$.    The \emph{orbital set} is defined to be
\[
\Gamma(C) = \bigcup_{g \in \Gamma} g(C).
\]
It is easy to see that if $C$ is (the hyperbolic closure of) a  fundamental domain, then the orbital set is the whole space, that is, $\Gamma(C) = \mathbb{D}^n$.  Moreover, the limit set is immediately contained in the Euclidean closure of any orbital set.  

There is a celebrated connection between hyperbolic geometry (especially Fuchsian groups) and the artwork of M. C. Escher. Orbital sets fall very naturally into this discussion since many of the memorable images from Escher's work  are orbital sets (rather than tessellations).  Here  $C$ could be a  large central bat or fish, which is then repeated many times on smaller and smaller scales towards the boundary of $\mathbb{D}^n$.

\subsection{Dimension theory} \label{dimsection}

There has been a great deal of interest in estimating the fractal dimension of the limit set of a Kleinian group.   We write $\hd, \, \ubd$ to denote the Hausdorff  and upper box dimension, respectively.  We refer the reader to \cite{falconer} for more background on dimension theory.  Since we use upper box dimension directly we recall the definition.  Given a bounded set $E$ in a metric space and a scale $\delta>0$, let $N_\delta(E)$ denote the smallest number of sets of diameter $\delta$ required to cover $E$.  (We say a collection of sets $\{U_i\}_i$ covers $E$ if $E \subseteq \cup_i U_i$.) Then the \emph{upper box dimension} of $E$ is 
\[
\ubd E = \limsup_{\delta \to 0} \frac{\log N_\delta(E)}{-\log \delta}.
\]
If we replace the $\limsup$ in the above with a $\liminf$ then we define the \emph{lower box dimension}, denoted by $\lbd E$. If the upper and lower box dimension coincide, then we write $\bd E$ for the common value and simply refer to the box dimension.  For all non-empty  bounded sets $E \subseteq \mathbb{R}^n$,
\[
0 \leq \hd F \leq \lbd E  \leq \ubd E \leq n.
\]
For all non-elementary Kleinian groups,  
\begin{equation} \label{dimlimit}
\delta(\Gamma) \leq \hd L(\Gamma)  \leq \lbd L(\Gamma)  \leq \ubd L(\Gamma)
\end{equation}
and for non-elementary \emph{geometrically finite}  Kleinian groups,  
\[
\delta(\Gamma) = \hd L(\Gamma) =  \bd L(\Gamma) .
\]
See \cite{bowditch} for more details on geometric finiteness.  In the geometrically infinite case, it is possible that $\delta(\Gamma) < \hd L(\Gamma)$.  These results go back to  Patterson  \cite{patterson}, Sullivan \cite{sullivan},  Bishop and Jones \cite{bishopjones} and Stratmann and Urba\'nski \cite{su}.  See the survey \cite{stratmann}.  Falk and Matsuzaki \cite{falk} characterise the upper box dimension of an arbitrary non-elementary Kleinian group as the  the convex core entropy of the group, denoted by $h_c(\Gamma)$.

Orbital sets provide a new family of fractal sets associated with Kleinian group actions. As such it is a well-motivated problem to consider their  dimension theory.  It is most natural to consider the dimensions of $\Gamma(C)$ with respect to the Euclidean metric on $\mathbb{R}^n$.  Since the Hausdorff dimension is countably stable and the maps $g \in \Gamma$ are conformal, it is immediate that
\begin{equation} \label{guess}
\hd \Gamma(C) = \hd C.
\end{equation}
The upper box dimension fails to be countably stable in general and so computing the upper box dimension of an orbital set is potentially an interesting problem.  Since upper box dimension is stable under taking closure, it is immediate that
\[
\ubd \Gamma(C) \geq \ubd L(\Gamma)
\]
and so we already see that the expected analogue of \eqref{guess} does not hold in general for upper box dimension.

\subsection{Inhomogeneous attractors}

The idea to study orbital sets is partially motivated by the theory of inhomogeneous iterated function systems, introduced by Barnsley and Demko \cite{barnsley}.  This is also where we took the name orbital set from.  We refer the reader to \cite{superfractals, barnsley, burrell, fraserinhom} for more details but briefly outline the connection here.  Consider an iterated function system (IFS), which is a finite collection of contractions $\mathcal{S} = \{S_i\}_i$ of a compact metric space $X$ into itself.  Fix a compact set $C \subseteq X$. A simple application of Banach's contraction mapping theorem yields that   there exists a unique non-empty compact attractor $F_C$ satisfying
\[
F_C = \bigcup_{S \in \mathcal{S}}  S(F_C)  \cup C.
\]
Classical attractors of IFSs are when $C=\emptyset$ and we denote these by $F_\emptyset$.  See \cite{falconer} for more background on classical IFS theory.  Let $\mathcal{M}$ be the monoid generated by $\mathcal{S}$.  It is straightforward to see that $F_C$ is the closure of the orbital set $\mathcal{M}(C)$. In particular, the box dimensions of $F_C$ and  $\mathcal{M}(C)$ coincide, but the Hausdorff dimensions may differ since countable stability guarantees $\hd \mathcal{M}(C) = \hd C$.  The  box dimensions of $F_C$ have been studied in several contexts, e.g. \cite{baker, burrell, fraserinhom,antti, lars}.  If the IFS $\mathcal{S}$ consists of similarity maps and satisfies the strong open set condition, then it was shown in \cite{fraserinhom} that
\begin{equation} \label{guess2}
\ubd F_C = \max\{ \ubd F_\emptyset, \ubd C\}.
\end{equation}
The lower box dimension does not behave so well, see \cite{fraserinhom}, and this formula does not necessarily hold when the strong open set condition fails, see \cite{baker}.

\section{Main results}

Our main result is a complete characterisation of the upper box dimension of Kleinian orbital sets with $C$ bounded in the hyperbolic metric.  This should be compared with \eqref{guess2}.  It is perhaps noteworthy that we do not make any assumptions on the Kleinian group. In particular, it does not have to be geometrically finite and the result  holds for elementary and non-elementary groups. We also do not require $C$ to be contained in a fundamental domain and so the images of $C$ appearing in the orbital set may overlap. We make essential use of the assumption that $C$ is bounded (in the hyperbolic metric) in the proof and it turns out that this  cannot be removed in general, see Theorem \ref{examplethm}.  For clarity we recall that we compute the dimension of the orbital set with respect to the Euclidean metric on $\mathbb{R}^n$.
\begin{thm}
\label{thm:2}
Let $\Gamma$ be a Kleinian group acting on $\mathbb{D}^n$ and $C$ be a non-empty bounded subset of $\mathbb{D}^n$. Then 
$$
\ubd \Gamma(C) = \max\left\{\ubd L(\Gamma), \ubd C, \delta(\Gamma)\right\}.
$$
\end{thm}

We defer the proof of Theorem \ref{thm:2} to Section \ref{proof1}. It is a new feature in the Kleinian groups case that three distinct terms appear in the maximum, recall \eqref{guess2}.  We briefly point out that all three terms are needed in general: 

\begin{enumerate}
	\item Suppose $\Gamma$ is geometrically infinite and satisfies $\delta(\Gamma) < \ubd L(\Gamma)$ and $C$ is a single point. Then
\[
\ubd L(\Gamma) > \delta(\Gamma) > 0 = \ubd C .
\]	
\item Suppose $\Gamma$ is generated by a single hyperbolic element and $C$ is a line segment.  Then
\[
\ubd C = 1 > 0 =   \delta(\Gamma) = \ubd L(\Gamma).
\]

	\item Suppose $\Gamma$ is generated by a single parabolic element and $C$ is a single point.  Then
\[
\delta(\Gamma) = 1/2 > 0 = \ubd L(\Gamma) = \ubd C.
\]
\end{enumerate}

Interestingly, the assumption that $C$ is a bounded subset of hyperbolic space $\mathbb{D}^n$ cannot be removed in general.  This was a surprise to us. 

\begin{thm} \label{examplethm}
There exists a non-empty set $C\subseteq \mathbb{D}^2$ and an (elementary) Fuchsian  group $\Gamma$  acting on $\mathbb{D}^2$ such that 
\[
\ubd L(\Gamma) =  \ubd C =  \delta(\Gamma) = 0
\]
but
\[
\lbd \Gamma(C) = \ubd \Gamma(C)= 1.
\]
(Theorem \ref{thm:2} ensures that such a set $C$ must be unbounded.)
\end{thm}
We defer the proof of Theorem \ref{examplethm} to Section \ref{proof2}.  If we make further assumptions about $\Gamma$, then one of the terms in the maximum from Theorem \ref{thm:2} may be dropped. The following two corollaries follow immediately from Theorem \ref{thm:2} together with well-known results, see the discussion in Section \ref{dimsection}.

\begin{cor}
Let $\Gamma$ be a geometrically finite, non-elementary Kleinian group  acting on $\mathbb{D}^n$ and  $C$ be a non-empty bounded  subset of $\mathbb{D}^n$. Then
$$
\ubd \Gamma(C)   = \max\left\{\ubd C,\delta(\Gamma)\right\}.
$$
\end{cor}

\begin{cor}
Let $\Gamma$ be a non-elementary Kleinian group  acting on $\mathbb{D}^n$  and  $C$ be a non-empty bounded subset of $\mathbb{D}^n$. Then
$$
\ubd \Gamma(C) = \max\left\{\ubd C,\ubd L(\Gamma)\right\}  = \max\left\{\ubd C,h_c(\Gamma)\right\} 
$$
where $h_c(\Gamma)$ is the convex core entropy of $\Gamma$.
\end{cor}

Finally, we note that we obtain simple bounds for the \emph{lower} box dimension of $\Gamma(C)$.  The upper bound uses the upper box dimension and the lower bound uses the fact that the lower box dimension is monotonic and stable under closure.  We get
\[
 \max\left\{\lbd C,\lbd L(\Gamma)\right\} \leq  \lbd \Gamma(C) \leq \max\left\{\ubd C,\ubd L(\Gamma)\right\}.
\]
This can be used to deduce that the box dimension of the orbital set exists in many cases, for example, if the box dimension of $C$ exists and $\Gamma$ is non-elementary and geometrically finite, then 
$$
\bd \Gamma(C)   = \max\left\{\bd C,\delta(\Gamma)\right\}.
$$
It would be interesting to consider the lower box dimension in greater detail, or the Assouad dimension, since both these dimensions are also not countably stable.  However, the lower box dimension is likely to behave very differently, based on \cite{fraserinhom}, in cases when the box dimension of $C$ does not exist.  The Assouad dimension of geometrically finite Kleinian groups was studied in \cite{fraserassouad} and it generally behaves differently from the upper box dimension in the case when there are parabolic elements.  The Assouad dimension of inhomogeneous self-similar sets was considered in \cite{antti}.

\section{Proof of Theorem \ref{thm:2}} \label{proof1}

Throughout the proof we write $A \lesssim B$ to mean there is a constant $c >0$ such that $A \leq cB$.  Similarly, we write $A \gtrsim B$ if $B \lesssim A$ and $A \approx B$ if $A \lesssim B$ and $A \gtrsim B$.  When the constant $c$ depends on a parameter $\theta$ we indicate this with a subscript (or multiple subscripts), e.g.  $A \lesssim_\theta B$.  The implicit constants will often depend on $\Gamma$,  $C$  and other fixed parameters, but it will be crucial that they never depend on the covering scale  $\delta>0$ used to compute the box dimension or on a specific element $g \in \Gamma$. 

Throughout the proof $B(x,R)$ will denote the closed Euclidean ball with centre $x \in \mathbb{D}^n$ and radius $R>0$.  We also write $|E|$ to denote the Euclidean diameter of a non-empty set $E$.

\subsection{Preliminary estimates and results from hyperbolic geometry}

We will frequently use the well-known estimate that, for all $g \in \conp$ and $z \in \mathbb{D}^n$,
\begin{equation} \label{diff}
|g'(z)| \approx_{|z|} 1-|g(z)|
\end{equation}
with implicit constants   independent of $g$. This estimate comes  directly from the definition of the hyperbolic metric and that $g$ is an isometry.

\begin{lma}
\label{comp}
For all  $r \in (0,1)$,  $z \in B(0,r)$ and $g \in\conp$,
\[
\frac{|g'(z)|}{|g'(0)|} \lesssim_r 1.
\]
\end{lma}
\begin{proof}
By \eqref{diff}, 
\[
\frac{|g'(z)|}{|g'(0)|} \approx_r  \frac{1-|g(z)|}{1-|g(0)|} \approx e^{d(0,g(0))-d(0,g(z))} \leq e^{d(g(0),g(z))} = e^{d(0,z)} \leq \frac{1+r}{1-r}
\]
proving the claim.
\end{proof}

\begin{lma}
\label{lemma:5}
Fix a non-empty bounded set $C \subseteq \mathbb{D}^n$.  Then
\[
 N_\delta(g(C))  \lesssim_C  N_{ \delta/ |g'(0)|}(C)
\]
for all  $\delta \in (0,1)$ and  $g \in \conp$.
\end{lma}
\begin{proof}
Let $r \in (0,1)$ be such that $C$ is contained in  $B(0,r)$.  We can choose such an $r$ depending only on $C$ since $C$ is bounded (in the hyperbolic metric). Let $\delta \in (0,1)$,   $g \in \conp$ and $\{U_i\}_i$ be a minimal $\big(\frac{\delta}{|g'(0)|}\big)$-cover of $C$.  We may assume that each $U_i \subseteq B(0,r)$.  Then for each $i$
\begin{align*}
|g(U_i)| \leq \frac{\delta}{|g'(0)|} \sup_{z \in U_i}|g'(z)| \leq \delta \sup_{z \in B(0,r)}\frac{|g'(z)|}{|g'(0)|} \lesssim_r \delta 
\end{align*}
 by Lemma \ref{comp}.  As such, $\{g(U_i)\}_i$ provides a $\lesssim_C \delta$ cover of $g(C)$ and the result follows.
\end{proof}

\begin{lma} \label{volume}
Let $r \in (0,1)$ and $\delta \in (0,1)$.  Suppose $\Gamma$ is a Kleinian group with a loxodromic element. If $g \in \Gamma$ is such that $|g(0)| \geq 1- \delta$, then $g(B(0,r))$ is contained in a $\lesssim_{r, \Gamma} \delta$ neighbourhood of the limit set.  
\end{lma}

\begin{proof}
Let $h \in \Gamma$ be loxodromic.  Loxodromic elements have precisely two fixed points on the boundary at infinity.  Let $z \in \mathbb{D}^{n}$ be a point lying on the (doubly infinite) geodesic ray joining the fixed points of $h$. We may assume that $h$ and $z$ are chosen to minimise $|z|$, which means $z$ depends only on $\Gamma$.  Then $g(z)$ lies on the geodesic ray joining the loxodromic fixed points of the loxodromic map $g h g^{-1}$.  These fixed points are the images of the fixed points of $h$ under $g$ and at least one of them must lie in the smallest Euclidean sphere passing through $g(z)$ and intersecting the boundary $S^{n-1}$ at right angles.  By  \eqref{diff} and applying Lemma \ref{comp}, the diameter of this  sphere is
\[
\lesssim  1-|g(z)| \lesssim_{r, \Gamma} 1-|g(0)| \leq \delta
\]
and the result follows. The dependency of the implicit constants on $\Gamma$  comes from the fact that in order to apply Lemma \ref{comp} we need to replace $r$ by $\max\{r, |z|\}$.
\end{proof}

\subsection{Proof of Theorem \ref{thm:2}}

\subsubsection{The lower bound}

Since upper box dimension is monotonic and stable under taking closure, it is immediate that
$$\ubd \Gamma(C) \geq \max\left\{\ubd C,\ubd L(\Gamma)\right\}.$$
Moreover, in the non-elementary case, $\ubd L(\Gamma) \geq \delta(\Gamma)$ giving the desired lower bound.  In the elementary case, $\delta(\Gamma) = 0$ unless $\Gamma$ is (freely) generated by finitely many parabolic elements sharing a single fixed point. If this is the case, then $\delta(\Gamma) = k/2$ where $k$ is the rank of  $\Gamma$.  In this case the lower bound $\ubd \Gamma(C) \geq  k/2 = \delta(\Gamma)$ follows since the orbit of a single point under $\Gamma$ is an inverted $k$-dimensional lattice.  It is a simple exercise to show that an inverted $k$-dimensional lattice has upper box dimension $k/2$. This completes the proof of the lower bound.

\subsubsection{The upper bound}

Let $t > \max\left\{\ubd C,\ubd L(\Gamma), \delta(\Gamma)\right\}$. Let $\delta \in (0,1)$.  We decompose  the orbital set   depending on how close the images $g(C)$ are  to the boundary. Images $g(C)$ which are close to the boundary will be small with respect to the Euclidean metric and images $g(C)$ which are far from the boundary will be large.  We then cover the close images  together, essentially just by covering the limit set, and we cover the large images separately.  More precisely, 
\begin{align*}
	N_\delta \left( \bigcup_{g \in \Gamma} g(C) \right) & \leq N_\delta \left( \bigcup_{\{g \in \Gamma : |g(0)| \leq 1 - \delta \}}g(C) \right) + N_\delta \left( \bigcup_{\{g \in \Gamma : |g(0)| \geq 1 - \delta \}}g(C) \right) \\ \\
	 & \leq \sum_{\{g \in \Gamma : |g(0)| \leq 1 - \delta \}} N_\delta(g(C)) + N_\delta \left( \bigcup_{\{g \in \Gamma : |g(0)| \geq 1 - \delta \}}g(C) \right).
\end{align*}

Consider the first term coming from the above decomposition.   First applying Lemma \ref{lemma:5} and then using $t> \ubd C$, we get

\begin{align*}
\sum_{\{g \in \Gamma : |g(0)| \leq 1 - \delta \}} N_\delta(g(C)) & \lesssim_C \sum_{\{g \in \Gamma : |g(0)| \leq 1 - \delta \}} N_{ \delta/ |g'(0)|}(C) \\ \\
&\lesssim_t \sum_{\{g \in \Gamma : |g(0)| \leq 1 - \delta \}}   \left( \frac{\delta}{|g'(0)|} \right) ^{-t}  \\ \\
&\leq \delta^{-t}  \sum_{g \in \Gamma} |g'(0)|^t \\ \\
&\lesssim_{t, \Gamma} \delta^{-t}
\end{align*}
since $t> \delta(\Gamma)$ using \eqref{diff}. In order to obtain the second estimate in the above it is crucial that  $ \delta/ |g'(0)| \lesssim 1$ for the $g$ we are summing over.  This is needed to apply the general box counting estimate for $C$.  However, this follows immediately from basic hyperbolic geometry since
$$
|g'(0)| = 1 - |g(0)|^2 \geq   1-|g(0)| \geq \delta.
$$

We now consider the second term in the original decomposition.  First suppose $\Gamma$ contains a loxodromic element.  Consider $g \in \Gamma$ such that $|g(0)| \geq 1-\delta$.  By Lemma \ref{volume}  $g(C) \subseteq g(B(0,r))$ is within $\lesssim_{C, \Gamma}  \delta$ of $L(\Gamma)$.  It follows that
\[
 \bigcup_{\{g \in \Gamma : |g(0)| \geq 1 - \delta \}}g(C)
\]
lies within a $\lesssim_{C,\Gamma} \delta$ neighbourhood of $L(\Gamma)$ and therefore
\[
N_\delta \left( \bigcup_{\{g \in \Gamma : |g(0)| \geq 1 - \delta \}}g(C) \right) \lesssim_{t,C,\Gamma} \delta^{-t}
\]
since $t > \ubd L(\Gamma)$.  This, combined with the estimate for the first term in the original decomposition, proves the upper bound in Theorem \ref{thm:2} in the case when $\Gamma$ contains a loxodromic element.  If $\Gamma$ does not contain a loxodromic element, then either it is a finite group and  $\ubd \Gamma(C) = \ubd C$ is immediate, or $\Gamma$ is (freely) generated by $k$ parabolic elements with a common fixed point.  In this case, it remains to establish
\[
\ubd \Gamma(C) \leq \max\{\ubd C,k/2\}.
\]
This can be achieved by a direct covering argument, but we present a slicker alternative.  It is known that $\delta(\Gamma') > k/2$ for a geometrically finite Kleinian group $\Gamma'$ which contains a free abelian subgroup of rank $k$ stabilising a parabolic fixed point. Moreover, this lower bound is sharp.  Therefore, we may find a sequence of geometrically finite  non-elementary groups $\Gamma_n$ ($ n \in \mathbb{N}$) each containing $\Gamma$ and with $\delta(\Gamma_n) \to k/2$.  Then for all $n$ the above argument gives
\[
\ubd \Gamma(C) \leq\ubd \Gamma_n(C)  =  \max\{\ubd C,\delta(\Gamma_n)\}
\]
and the result follows. One can even explicitly construct $\Gamma_n$ by choosing $\Gamma_n = \langle \Gamma, h^n\rangle$ where $h$ is a loxodromic element which does not fix the common parabolic fixed point of $\Gamma$ and $n$ is sufficiently large.

\section{Proof of Theorem \ref{examplethm}} \label{proof2}

The set $C$ and group $\Gamma$ are very simple.  The work is in proving that the orbital set has large dimension, and this relies on some number theory.  Let $\alpha>1$ and $\beta \in (0,1)$ be such that $\log \alpha$ and $\log \beta$ are rationally independent.  Here and throughout $\log$ is the natural logarithm.  For example, $\alpha = 2$, and $\beta = 1/3$ suffices.  Let
\[
C=\{ 1-\beta^n : n \in \mathbb{N}\} \subseteq \mathbb{D}^2
\]
noting that $C$ is unbounded in $(\mathbb{D}^2, d)$.  Let $h \in \textup{con}^+(\mathbb{D}^2)$ be the hyperbolic element with repelling fixed point $-1$ and attracting fixed point $1$ given by
\[
h(z) = \frac{(\alpha+1)z + (\alpha-1)}{(\alpha-1)z + (\alpha+1)}.
\]
Let $\Gamma = \langle h \rangle$ be the elementary Fuchsian group generated by $h$.  By construction
\[
\ubd L(\Gamma) =  \ubd C =  \delta(\Gamma) = 0.
\]
In order to prove that $\lbd \Gamma(C) = \ubd \Gamma(C)= 1$ we show that $\Gamma(C)$ is dense in $(-1,1) \subseteq \mathbb{D}^2$, recalling that the box dimensions are stable under taking closure. This shows that the   box dimesions are at least 1, but the orbital set is contained in $(-1,1)$ and so they are also at most 1.  The orbital set has a straightforward description due to the simplicity of $\Gamma$ and $C$:
\begin{align*}
\Gamma(C) &= \{ h^m( 1-\beta^n) : m \in \mathbb{Z}, \, n \in \mathbb{N}\} \\ 
&= \left\{ \frac{(\alpha^m+1)( 1-\beta^n) + (\alpha^m-1)}{(\alpha^m-1)( 1-\beta^n) + (\alpha^m+1)} : m \in \mathbb{Z}, \, n \in \mathbb{N}\right\} \\ 
&= \left\{ \frac{ 2-\alpha^m\beta^n -\beta^n}{2+\alpha^m\beta^n -\beta^n} : m \in \mathbb{Z}, \, n \in \mathbb{N}\right\}
\end{align*}
noting that we switch the role of $m$ and $-m$ in the final expression, which is fine since $m \in \mathbb{Z}$. Let $y \in (0,\infty)$ be such that $\log y \in \mathbb{Q}$. Since $\log \alpha/\log \beta \notin \mathbb{Q}$ we can find sequences $m_k \in \mathbb{Z}, \, n_k \in \mathbb{N}$ such that
\[
\alpha^{m_k}\beta^{n_k} \to y
\]
as $k \to \infty$.  This is a standard application of Dirichlet's approximation theorem.  Moreover, since  $\log \alpha$ and $\log \beta$ are rationally independent we necessarily have $n_k \to \infty$ as $k \to \infty$.  Therefore
\[
\frac{ 2-\alpha^{m_k}\beta^{n_k} -\beta^{n_k}}{2+\alpha^{m_k}\beta^{n_k} -\beta^{n_k}} \to \frac{2-y}{2+y}
\]
as $k \to \infty$.  The set
\[
\left\{ \frac{2-y}{2+y} : y \in (0,\infty) \text{ and } \log y \in \mathbb{Q}\right\}
\]
is dense in $(-1,1)$ and the density of $\Gamma(C)$ in $(-1,1)$ follows.

\end{document}